\documentclass[10pt]{article}
\usepackage[utf8]{inputenc}
\usepackage{amsmath,amssymb,amsthm}
\usepackage{mathtools}
\usepackage{enumerate}
\usepackage{tikz}
\usetikzlibrary{arrows.meta,calc}
\usepackage{lipsum}

\usepackage{graphicx} 
\usepackage{tikz-3dplot}
\usepackage{comment}

\newtheorem{theorem}{Theorem}[section]
\newtheorem{lemma}[theorem]{Lemma}
\newtheorem{definition}[theorem]{Definition}
\newtheorem{example}[theorem]{Example}

\newtheorem{notation}[theorem]{Notation}
\newtheorem{remark}[theorem]{Remark}
\newtheorem{corollary}[theorem]{Corollary}
\newtheorem{theorem/definition}[theorem]{Satz/Definition}

\newcommand{\R}{\mathbb R}

\newcommand{\N}{\mathbb{N}}

\newcommand{\C}{{\cal C}}

\newcommand{\op}{\operatorname}
\newcommand{\supp}{\operatorname{supp}}

\newcommand{\ep}{\varepsilon}
\renewcommand{\(}{\left(}
\renewcommand{\)}{\right)}

\newcommand{\overbar}[1]{\mkern 1.5mu\overline{\mkern-1.5mu#1\mkern-1.5mu}\mkern 1.5mu}

\numberwithin{equation}{section}

\begin{document}

\begin{center}
{\large\bfseries
Semiconvexity of (weak) Kantorovich potentials in the Lorentzian optimal transport problem
}

\vspace{1em}

{
{\sc Alec Metsch} \\
Universität zu Köln \\
email: ametsch@math.uni-koeln.de
}

\vspace{1em}

{
}
\end{center}

\begin{abstract}
    \noindent
    We study semiconvexity properties of (weak) Kantorovich potentials for the Lorentzian optimal transport problem with the standard cost function $c$. We show that, in general, this regularity – known in the Riemannian context – does not extend to the Lorentzian setting. Nevertheless, we provide a general regularity result for $c$-convex functions and show that, under suitable general assumptions on the measures, this yields semiconvexity of the (weak) potentials at least on an open set of full measure. This, in turn, allows us to conclude the existence and uniqueness of an optimal transport map.
\end{abstract}

\tableofcontents

\section{Introduction}

The quadratic Monge problem for two Borel probability measures $\mu$ and $\nu$ on $\R^n$ consists of minimizing
\begin{align}
    \inf\bigg\{\int |x-T(x)|^2\, d\mu(x)\mid T_\#\mu = \nu\bigg\}  \label{eqcb}
\end{align}
among all transport maps $T$, i.e.\ Borel maps $T:\R^n\to \R^n$ such that $T_\#\mu=\nu$. Here, $T_\#\mu$ denotes the pushforward measure of $\mu$ w.r.t.\ $T$. Brenier's celebrated theorem \cite{Brenier} states that, if $\mu$ is absolutely continuous w.r.t.\ the Lebesgue measure and the total cost \eqref{eqcb} is finite, Monge's problem admits a unique solution. Moreover, the optimal transport map $T$ is given by the gradient of a convex function. The standard proof proceeds by first finding an optimal coupling $\pi$ for the Kantorovich problem
\begin{align*}
    \inf\bigg\{\int |x-y|^2\, d\pi(x,y)\mid \pi\in \Gamma(\mu,\nu)\bigg\},
\end{align*}
and then by constructing – via the Rockafellar method – some convex function $\varphi$ satisfying
\[
    \varphi^*(y) +\varphi(x) = \langle x,y\rangle \quad \text{ for } \pi\text{-a.e. $(x,y)$,}
\]
$\varphi^\ast$ denoting the Legendre transform. Since $\varphi^\ast(y)+\varphi(x)\geq \langle x,y\rangle$, differentiating the above idenity yields $y=T(x):=\nabla \varphi(x)$ $\pi$-a.e., showing that $\pi$ is induced by the map $T$ – and thus unique. Convexity of $\varphi$ is crucial, as it ensures differentiability $\mu$-a.e.

Similarly, when the cost function $c$ is given by the squared Riemannian distance $d^2$ on a complete Riemannian manifold and $\pi$ is an optimal coupling with finite cost, the standard Rockafellar method provides a $d^2$-convex function $\varphi$ which satisfies 
\begin{align}
    \varphi^{d^2}(y)-\varphi(x) = d^2(x,y) \quad \text{ for $\pi$-a.e.\ $(x,y)$.} \label{eqbt}
\end{align}
Under suitable integrability assumptions, $(\varphi,\varphi^{d^2})$ solves the dual problem, and $\varphi$ is just a Kantorovich potential. If $\mu$ is absolutely continuous w.r.t.\ the volume measure, it was shown by McCann \cite{McCann} in the compact and by Figalli/Gigli \cite{Figalli/Gigli} in the general case that $\varphi$ inherits semiconvexity from the semiconcavity of the cost (McCann relied on Lipschitz continuity instead of semiconvexity) – at least on the “region of interest,” namely an open set of full $\mu$-measure. As in Brenier’s theorem, we may then differentiate \eqref{eqbt} to conclude the existence and uniqueness of an optimal transport map.
The same method also applies to certain classes of sufficiently regular \emph{real valued} cost functions \cite{Villani}.

More recently, originating from \cite{Eckstein/Miller}, the Lorentzian optimal transport problem has attracted considerable attention (see e.g.\, \cite{Bertrand/Puel/Pratelli,Bertrand/Puel, Kell, McCann2, Mondino/Suhr, Suhr}). However, the standard Lorentzian cost function $c$ (cf.\ \eqref{eqh}; some authors also consider $p=1$) presents substantial additional challenges, as it is significantly less regular than the squared Riemannian distance; $c$ is not even real-valued. This causes many classical results, including the standard Rockafellar construction, to break down in general.

Nevertheless, in \cite[Theorem 2.8]{Kell}, Kell and Suhr establish fairly general conditions on the measures under which a modified Rockafellar construction remains valid: for any optimal coupling $\pi$, there exists a $c$-convex function $\varphi$ such that
\[
    \varphi^c(y)-\varphi(x) = c(x,y) \quad \text{for $\pi$-a.e.\ $(x,y)$.}
\]
We refer to such a function $\varphi$ as a $\pi$-solution, or a (weak) Kantorovich potential (see Definition \ref{def4}). Under suitable integrability assumptions, the pair $(\varphi, \varphi^c)$ solves the dual problem, and $\varphi$ becomes an actual Kantorovich potential.

A natural question is whether the semiconvexity results for $\pi$-solutions (or, more generally, $c$-convex functions) extend from the Riemannian to the Lorentzian setting. Again the little regularity and non-finiteness of the cost $c$ provides many difficulties that do not arise in the Riemannian framework. Up to this point, semiconvexity has only been proven for very special situations; for instance if the pair $(\mu,\nu)$ belongs to the now well-known class of $p$-separated measures \cite{McCann2}. In these cases, semiconvexity comes as a simple consequence of the well-chosen definitions. 

The main goal of this paper is therefore to provide quite general assumptions on the measures under which any $\pi$-solution is semiconvex on the region of interest (Theorem \ref{thmc}). Combined with the assumptions of Kell/Suhr, this result yields – as in Brenier's theorem – the existence and uniqueness of an optimal transport map (Corollary \ref{cor2}).
As in the Riemannian setting, Theorem \ref{thmc} follows from a general regularity result for $c$-convex functions, which we establish. However, one cannot expect our result to be as general as its Riemannian counterpart: We provide a simple example showing that $c$-convex functions (and even Kantorovich potentials) may fail to be continuous on the interior of their domain – in contrast to the Riemannian setting.

\subsection{Setting and results}

To state our precise results, let $(M,g)$ be a smooth globally hyperbolic spacetime of dimension $n+1$, and fix $0<p< 1$. We consider the cost function
\begin{align}
    c(x,y):=
    \begin{cases}
        -d^p(x,y),\ &\text{if } (x,y)\in J^+,
        \\\\
        +\infty,\ &\text{else.}
    \end{cases} \label{eqh}
\end{align}
Here, $d$ denotes the Lorentzian distance function and $J^+$ the set of causally related pairs.
We study the optimal transport problem associated with this cost function, that is,
\begin{align*}
    C(\mu,\nu)=\inf\bigg\{\int_{M\times M} c(x,y)\, d\pi(x,y)\mid \pi \in \Gamma_{\leq}(\mu,\nu)\bigg\}\in \R\cup\{\pm \infty\},
\end{align*}
where $\mu,\nu\in {\cal P}:={\cal P}(M)$ are Borel probability measures on $M$,  $\Gamma_{\leq}(\mu,\nu)$ denotes the set of causal couplings of $\mu$ and $\nu$ (i.e.\ all Borel probability measures $\pi$ on $M\times M$ having $\mu$ and $\nu$ as marginals and such that $\pi(J^+)=1$). We adopt the convention that $\inf(\emptyset):=+\infty$.
\medskip

Rather than working with Kantorovich potentials directly, we introduce and use a weaker notion, which we call a $\pi$-solution. The motivation for this choice is twofold. First, general existence results for Kantorovich potentials typically rely on integrability conditions on the cost function that are not satisfied in the present setting. Second, for the purpose of proving existence and uniqueness of optimal transport maps, it is sufficient to work with these weaker objects.

\begin{definition}[$\pi$-solutions]\rm \label{def4}
    Let $\mu,\nu\in {\cal P}$ and $\pi\in \Gamma_\leq(\mu,\nu)$. We say that a function $\varphi:M\to \overbar \R$, $c$-convex relatively to $(M,\supp(\nu))$, is a \emph{$\pi$-solution} provided 
    \begin{align}
        \pi(\partial_c\varphi)=1. \label{eqaa}
    \end{align}
\end{definition}
For the definition of a $c$-convex function and of $\partial_c\varphi$, see Definition \ref{def1}.

\begin{definition}[Causal compactness]\rm \label{causally}
    A subset $A\subseteq M$ is called \emph{causally compact} if, for any compact set $K\subseteq M$, the sets $J^+(K)\cap A$ and $J^-(K)\cap A$ are compact.
\end{definition} 

We are now in a position to state our main result. We denote by $\operatorname{vol}_g$ the volume measure on $M$ with respect to $g$. The key ingredient in the proof of Theorem \ref{thmc} is Corollary \ref{corasd} which concerns regularity of quite general $c$-convex functions. Our argument builds on and extends the strategy of Theorem 2.12 in \cite{Kell}.

In the following, we write $I^+$ for the set of chronologically related pairs.

\begin{theorem}[Semiconvexity of $\pi$-solutions]\label{thmc}
    Let $\mu,\nu\in {\cal P}$. Suppose that
    \begin{enumerate}[(a)]
        \item $\supp(\nu)$ is causally compact
        \item $\mu\ll \op{vol}_g$,
        \item $\supp(\mu)\cap \supp(\nu)=\emptyset$.
    \end{enumerate}
    Suppose that $\pi\in \Gamma_\leq(\mu,\nu)$ admits a $\pi$-solution $\varphi:M\to \overbar \R$. Then $\pi(I^+)=1$ and $\varphi$ is locally semiconvex on an open set of full $\mu$-measure.
\end{theorem}

Assumption (b) is as in Brenier's theorem. To illustrate why assumptions (a) and (c) are crucial, we refer to the discussion in Subsection \ref{method}. There, we also address the additional difficulties that arise in the Lorentzian framework as compared to the Riemannian setting.
\medskip

The fact that an optimal coupling $\pi$ must be concentrated on $I^+$ was proved in the special case $p=1$ by Suhr and Kell in \cite{Kell}, Theorem 2.12, under the assumption that $\mu\ll \op{vol}_g$, $\supp(\mu)\cap \supp(\nu)=\emptyset$ and that $\pi$ admits a Borel $\pi$-solution. The proof readily extends to the case $p\in (0,1]$, and it implies that $\pi$-solutions are approximately differentiable almost everywhere. Similar to the reasoning in \cite{Fathi/Figalli} or Theorem 10.38 in \cite{Villani}, this can be used to prove that the optimal coupling is unique and characterized by 
\[
    \widetilde \nabla \varphi(x) +\nabla_x c(x,y) = 0 \text{ $\pi$-a.e.}
\]
In particular, $\pi$ is induced by a map. However, it is interesting to know when the approximate gradient can be replaced by an actual gradient, which explains the interest of Theorem \ref{thmc}.
\medskip

If an optimal coupling $\pi$ admits a $\pi$-solution, Theorem \ref{thmc} together with standard arguments guarantees the existence of an optimal transport map. Uniqueness, however, is less straightforward: classically, it is proved by showing that each optimal coupling is induced by a map, which follows from Theorem \ref{thmc} only if every optimal coupling admits a $\pi$-solution. Fortunately, Theorem 2.8 (or, more precisely, its proof) in the work of Kell and Suhr \cite{Kell} shows that this holds under the following conditions:
\begin{itemize}
    \item $\supp(\mu)$ is connected, 
    \item $\mu$ and $\nu$ are \emph{strictly timelike}, i.e.\ there exists a (not necessarily optimal) coupling $\pi\in \Gamma(\mu,\nu)$ with $\supp(\pi)\subseteq I^+$,
    \item $\mu,\nu\in {\cal P}_c$ (${\cal P}_c$ denotes the probability measures with compact support).
\end{itemize}

The following corollary is thus immediate:

\begin{corollary}[Existence, uniqueness and characterization of an optimal map]\label{cor2}
    Let $\mu,\nu\in {\cal P}_c$ satisfy the assumptions of Theorem \ref{thmc} and the conditions above. Then the total cost $C(\mu,\nu)$ is finite, there exists a unique optimal coupling $\pi$, it is concentrated on $I^+$, and there exists a $\pi$-solution $\varphi$ such that $\pi$ is characterized by the equation
    \begin{align}
        \nabla \varphi(x) +\nabla_x c(x,y)=0\ \pi\text{-a.e.} \label{eqbp}
    \end{align}
    In particular, $\pi$ is induced by the map $T(x):=\nabla_x c(x,\cdot)^{-1}(-\nabla \varphi(x))$.
\end{corollary}

Uniqueness of an optimal coupling among all couplings concentrated on $I^+$ holds under much weaker assumption (\cite{McCann2}, Theorem 7.1), where it is also shown that this coupling is induced by a map. The new part of this theorem is the characterization of the optimal coupling via equation \eqref{eqbp} (with an actual derivative). It is also interesting to note that the optimal coupling above is in general only concentrated, but not supported on $I^+$ (see Section \ref{secd}).

\subsection{Relevance of semiconvexity for Ricci curvature bounds}

Note that semiconvexity of Kantorovich potentials for the squared Riemannian distance was crucial in the celebrated characterisation of lower Ricci curvature bounds via convexity properties of the Boltzmann-Shannon entropy along Wasserstein geodesics (introduced by McCann in \cite{McCann3}) \cite{Lott/Villani, Sturm1, Sturm2} – which laid the foundation for the rapidly developing synthetic theory of lower Ricci curvature bounds in the setting of metric measure spaces \cite{Ambrosio/Mondino, Ambrosio/Gigli/Savare2, Ambrosio/Gigli/Savare1, Ambrosio/Mondino/Savare, Erbar/Kuwada/Sturm, Lott/Villani, Sturm1, Sturm2}. Indeed, it involves a nontrivial computation of the Jacobian of the optimal transport map along the transport geodesics, which crucially relies on the semiconvexity of the Kantorovich potential.

Similarly, the characterization of Ricci bounds was later extended to (globally hyperbolic) spacetimes independently by McCann \cite{McCann2} and Mondino/Suhr \cite{Mondino/Suhr}. As in the Riemannian setting, this led to the development of a synthetic theory of lower Ricci bounds for Lorentzian pre-length spaces \cite{Braun, Mondino/Cavalletti,Kunzinger/Sämann} (see also \cite{Octet,Minguzi/Suhr}). McCann's computations in \cite{McCann2} for this characterization also rely on the semiconvexity of the potentials. 
However, he only considered the case of "$q$-separated" measures – a very strong condition – as this was sufficient for his purposes and general semiconvexity results fail in broader settings.

As explained above, in this work, we aim to go beyond these restrictive assumptions by considering more general measures.

\subsection{Method and difficulties}\label{method}

As indicated above, in the Riemannian setting \cite{Figalli/Gigli}, the semiconvexity of the Kantorovich potential on the region of interest follows from the general result stating that any $d^2$-convex function $\varphi$ (with $d$ the Riemannian distance) is locally semiconvex on the interior of its domain $\op{dom}(\varphi):=\{\varphi\in \R\}$, and that the boundary of the domain is countably $(n-1)$-rectifiable ($n$ being the dimension). Indeed, if $\mu\ll \op{vol}_g$, then a region of interest is just $\op{int}(\op{dom}(\varphi))$. 

This approach is based on the methods by Villani \cite{Villani}, where the analogous result is proved for certain classes of cost functions satisfying sufficient regularity conditions. Similar techniques have been applied by Gangbo and McCann \cite{Gangbo/McCann} and McCann \cite{McCann}.
 Let us briefly outline the general approach developed in \cite{Figalli/Gigli, Villani}, whose fundamental arguments we will also follow:
\medskip

Let $M$ be an $n$-dimensional Riemannian manifold, $d$ the Riemannian distance, and $\varphi:M\to \R\cup\{\infty\}$ a $d^2$-convex function. For some function $\psi$,
\begin{align}
    \varphi(x) = \sup\{\psi(y)-d^2(x,y)\mid y\in M\} \quad \forall x\in M.
    \label{eqca}
\end{align}
Setting $\Omega:=\op{Int}(\op{dom}(\varphi))$, the strategy is to show that $\partial\Omega$ is countably $(n-1)$-rectifiable and that the $c$-subdifferential $\partial_c\varphi(U)$ is precompact for any $U\Subset \Omega$. Once this is established, for $x\in U$, we can restrict the supremum in \eqref{eqca} to the compact set $\overbar{\partial_c\varphi(U)}$; it then follows easily that $\varphi$ inherits local semiconvexity on $U$ (and hence on $\Omega$) from $-d^2$. 
\medskip

If $M$ is a globally hyperbolic spacetime of dimnesion $n+1$, $c$ is the cost function from \eqref{eqh}, and $\varphi:M\to \R\cup\{\pm\infty\}$ is any $c$-convex function, then, as in the Riemannian case, we have
 \begin{align}
    \varphi(x) = \sup\{\psi(y)-c(x,y)\mid y\in M\} \quad \forall x\in M
    \label{eqcaa}
\end{align}
for some $\psi$. It is not difficult to verify that the boundary of $\op{dom}(\varphi)$ is countably $n$-rectifiable (Lemma \ref{Figalli}).
However, in contrast to the Riemannian setting, $\varphi$ is, in general, not locally semiconvex and not even continuous on $\Omega:=\op{int}(\op{dom}(\varphi))$.

The first problem is that there exist examples showing that, in general, for $x\in U\Subset \Omega$, the supremum in \eqref{eqcaa} cannot be restricted to a compact set (Example \ref{remc}(b)) – this is a consequence of the lack of superlinearity of $c$.
For this reason, we impose the assumption of causal compactness on the target measure $\nu$: since we are only interested in $c$-convex functions which arise as $\pi$-solutions in the optimal transport problem, it suffices – provided $\supp(\nu)$ is causally compact – to consider the case where
 \begin{align}
    \varphi(x) = \sup\{\psi(y)-c(x,y)\mid y\in A\} \quad \forall x\in M.
    \label{eqcaaa}
\end{align}
for some causally compact set $A\subseteq M$ (in the optimal transport problem, one has $A=\supp(\nu)$). In other words, $\varphi$ is $c$-convex relativley to $(M,A)$.
Then for any $x\in U\Subset \Omega$, the supremum in \eqref{eqcaaa} can be restricted to the compact set $J^+(\overbar U)\cap A$.

The main difficulty in this paper is that $c$ is not locally semiconcave (or even real-valued) on $M\times M$; as a consequence, even the supremum in \eqref{eqcaaa} is, in general, not even continuous on $\Omega$, see Example \ref{remc}(a). However, $c$ is locally semiconcave on the open set of chronologically related points $I^+$. The goal is therefore to prove that there exists an open subset $\Omega_0\subseteq \Omega$ with $\op{vol}_g(\Omega\backslash \Omega_0)=0$ such that, for every $x\in U\Subset \Omega_0$, the supremum in \eqref{eqcaaa} may be restricted to a set of the form $\{y\in A\mid d(x,y)\geq \ep\}$, where $\ep>0$. In other words, the supremum must be locally uniformly bounded away from the complement of the chronological future. Once this is established, local semiconvexity on $\Omega_0$ (the region of interest if $\mu\ll \op{vol}_g$) follows as in the classical case. Assumption (c) in Theorem \ref{thmc} is required to guarantee this property.  

\subsection{Plan of the paper}

In Section \ref{sec2a}, we recall some fundamental results from Lorentzian geometry and optimal transport theory. Section \ref{secb} provides examples showing that well-known Riemannian regularity results do not extend to the Lorentzian setting. Theorem \ref{thmc}, together with an adapted general regularity result for $c$-convex functions, is proved in Section \ref{secc}. In Section \ref{secd}, we prove Corollary \ref{cor2}.

\section{Preliminaries and notation}\label{sec2a}

\textbf{For the rest of this paper}, let $p\in (0,1)$ and let $(M,g)$ be an $(n+1)$-dimensional globally hyperbolic spacetime, where the metric $g$ is taken to have signature $(-,+,...,+)$. In particular, $M$ is time-oriented. We refer to future-directed causal (resp.\ timelike) vectors simply as causal (resp.\ timelike). A curve is always assumed to be piecewise smooth if not otherwise said. In particular, a curve is referred to as causal (timelike) if it is piecewise smooth and future-directed causal (timelike). We denote by $I^+$ resp.\ $J^+$ the set of pairs $(x,y)$ which can be connected by a timelike resp.\ causal curve (with $0$ being a causal vector). The Lorentzian distance function is denoted by $d$:
\[
    d(x,y):=
    \begin{cases}
        \sup_\gamma {\ell}_g(\gamma),&\text{ if } (x,y)\in J^+,
        \\
        0,& \text{ else,}
    \end{cases}
\]
where the supremum is taken over all causal curves $\gamma:[a,b]\to M$ connecting $x$ to $y$ and the (Lorentzian) length ${\ell}_g(\gamma)$ of a causal curve $\gamma:[a,b]\to M$ is defined as
\[
    \int_a^b \sqrt{|g(\dot \gamma(t),\dot \gamma(t))|}\, dt.
\]
For $x\in M$, we denote by $\C_x\subseteq T_xM$ the cone of causal vectors. Note that $\C_x$ is closed. We also set $\C:=\{(x,v)\in TM\mid v\in \C_x\}$. We can equip $M$ with a complete Riemannian metric, which will be fixed and denoted by $h$. All balls $B_r(x)$, $x\in M$, $r>0$, are understood to be taken w.r.t.\ the metric $h$. The $h$- and $g$-norm of a tangent vector $v\in T_xM$ are denoted by $|v|_h$ and $|v|_g:=\sqrt{|g(v,v)|}$, respectively.

\subsection{Properties of the Lorentzian distance function}

\begin{notation} \rm
    We reserve the term \emph{maximize} to refer specifically to the Lorentzian length functional. That is, a maximizing curve $\gamma:[a,b]\to M$ is a causal curve that satisfies $\ell_g(\gamma)=d(\gamma(a),\gamma(b))$. The following result is well-known.
\end{notation}

\begin{theorem}[Properties of the distance function]\label{asdfghj}
    For any two points $(x,y)\in J^+$, there exists a maximizing geodesic connecting $x$ to $y$. Moreover, any maximizing curve must be a pregeodesic.
\end{theorem}
\begin{proof}
    See \cite{ONeill}, Chapter 14, Proposition 19 and \cite{Minguzzi}, Theorem 2.9.
\end{proof}

Recall that a function defined on an open subset of a smooth manifold is called \emph{locally semiconvex/semiconcave} if it is so when computed in local coordinates.
The following theorem asserts that $d$ is locally semiconvex on the chronological future $I^+$. 

\begin{theorem}[Semiconvexity of the distance function]\label{A}
    The Lorentzian distance function is real-valued and continuous on $J^+$, and locally semiconvex on $I^+$. In particular, $c$ is locally semiconcave on $I^+$.
\end{theorem}
\begin{proof}
	See \cite{McCann2}, Proposition 3.4, or \cite{Metsch3} for the case $p=\frac 12$.
\end{proof}

\subsection{$c$-convex functions and $\pi$-solutions}

	In this section, we recall some basic definitions and notations related to $c$-convex functions and introduce the notion of a $\pi$-solution.

    The following definition is adapted from \cite{Mondino/Cavalletti}, Definition 2.22.

\begin{definition}[$c$-convexity]\rm \label{def1}
    Let $U,V\subseteq M$. A function $\varphi:U\to \overbar \R$ is said to be \emph{$c$-convex relatively to $(U,V)$} if there exists a function $\psi:V\to \overbar \R$ such that 
    \begin{align}
        \varphi(x)=\sup_{y\in V} \(\psi(y)-c(x,y)\) \text{ for all } x\in U. \label{eqbs}
    \end{align}
    Here, we use the convention $+\infty-\infty:=-\infty$.
    The $c$-transform of $\varphi$ is the function 
    \begin{align}
        \varphi^c:V\to \overbar \R,\ \varphi^c(y):=\inf_{x\in U} (\varphi(x)+c(x,y)). \label{qwedfghjki}
    \end{align}
    Here, we use the convention $-\infty+\infty:=+\infty$. It is well-known that \eqref{eqbs} holds with $\varphi^c$ instead of $\psi$ (\cite{Ambrosio}, Definition 6.1.2).
    The $c$-subdifferential of $\varphi$ is defined as
    \[
        \partial_c\varphi:=\{(x,y)\in U\times V\mid \varphi(x)\in \R,\ \varphi^c(y)\in \R,\ \varphi^c(y)-\varphi(x)=c(x,y)\}.
    \]
 	Note that $c(x,y)<\infty$ whenever $(x,y)\in \partial_c\varphi$.
\end{definition}

\begin{remark}\rm
    If $\varphi$ is $c$-convex relatively to $(U,V)$, then $\varphi$ can be extended to a $c$-convex function relatively to $(M,M)$. Indeed, let $\psi:V\to \overbar \R$ be as in the definition above, and extend $\psi$ to $M\backslash V$ by $-\infty$. If we denote this extension by $\psi'$, we can define the $c$-convex function
    \[
        \varphi'(x):=\sup_{y\in M} (\psi'(y)-c(x,y))=\sup_{y\in V} (\psi'(y)-c(x,y)),\ x\in M.
    \]
    
    Conversely, if $\varphi:M\to \overbar \R$ is $c$-convex relatively to $(M,V)$, then $\varphi_{|U}$ is $c$-convex relatively to $(U,V)$ for any $U\subseteq M$. 

    Therefore, we will only consider $c$-convex functions relatively to $(M,V)$. Sometimes we will just speak informally of a $c$-convex function and mean a $c$-convex function relatively to $(M,V)$ for some $V\subseteq M$.
\end{remark}

Recall the definition of a $\pi$-solution: 
\begin{definition}[$\pi$-solutions]\rm
     Let $\mu,\nu\in {\cal P}$ and $\pi\in \Gamma_\leq(\mu,\nu)$. We say that a function $\varphi:M\to \overbar \R$, $c$-convex relatively to $(M,\supp(\nu))$, is a \emph{$\pi$-solution} provided 
    \begin{align*}
        \pi(\partial_c\varphi)=1. 
    \end{align*}
\end{definition}

\section{Riemann vs Lorentz – Counterexamples}\label{secb}

As discussed in the introduction, for a well-behaved cost function $\tilde c$ (as in \cite{Figalli/Gigli,Villani}) defined on a Riemannian manifold, any $\tilde c$-convex function $\varphi$ is locally semiconvex on $\Omega := \operatorname{int}({\varphi < \infty})$, and its boundary $\partial \Omega$ is countably $(n-1)$-rectifiable. In \cite{Figalli/Gigli,Villani}, this result is derived from the local semiconcavity of $\tilde c$ together with the compactness of the $\tilde c$-subdifferential $\partial_{\tilde c} \varphi(K)$ for every compact set $K \Subset \Omega$.

We show that this result does not extend so generally to the Lorentzian setting. Specifically, we construct examples of simple $c$-convex functions that (i) fail to be continuous in the interior of their domain, and (ii) have an unbounded $c$-subdifferential as $x$ ranges over a compact subset of $\Omega$.

Before presenting these counterexamples, we establish the following lemma, which also constitutes the first step in \cite{Figalli/Gigli, Villani}. Notably, this is one of the few instances in which the Lorentzian framework turns out to be simpler than the Riemannian one.

\begin{lemma}[First properties of $c$-convex functions]\label{Figalli}
    Let $\varphi:M\to \overbar \R$ be a $c$-convex function relatively to $(M,V)$, $V\subseteq M$. Set $\op{dom}(\varphi)=\{x\in M\mid \varphi(x)\in \R\}$. Then $\varphi$ is locally bounded on $\op{int}(\op{dom}(\varphi))$ and $\partial\op{dom}(\varphi)$ is countably $n$-rectifiable.
\end{lemma}
\begin{proof}
     Note that $\partial \op{dom}(\varphi)=\partial D^+\cup \partial D^-$ where 
    \[
        D^\pm:=\{x\in M\mid \varphi(x)\in \R\cup\{\pm \infty\}\}.
    \]
    For $x,x',y\in M$ with $x'\in J^+(x)$, the following implication holds:
    \[
        y\in J^+(x')\Rightarrow y\in J^+(x) \text{ and } c(x,y) \leq c(x',y).
    \]
    Hence, denoting $\psi:=\varphi^c$, we have
    \begin{align}
        \varphi(x) = \sup_{y\in J^+(x)} \Big(\psi(y)-c(x,y)\Big) \geq \sup_{y\in J^+(x')} \Big(\psi(y)-c(x',y)\Big) = \varphi(x'). \label{eqbf}
    \end{align}
    Now, if $x\in \op{int}(\op{dom}(\varphi))$, choose points $x^\pm \in I^\pm(x)$ such that $\varphi(x^\pm)\in \R$. For $x'$ close to $x$, we have $x'\in I^+(x^-)\cap I^-(x^+)$, and thus
    \[
        \varphi(x') \in   (\varphi(x^+),\varphi(x^-))\subseteq \R.
    \]
    This shows that $\varphi$ is locally bounded on $\op{int}(\op{dom}(\varphi))$. 
    
    Next, if $x\in \partial D^+$, then $I^+(x)\subseteq \varphi^{-1}(-\infty)$. Indeed, let $x'\in I^+(x)$, and let $(x_k)$ be a sequence converging to $x$ with $\varphi(x_k)=-\infty$ for all $k$. Then $x'\in I^+(x_k)$ for large $k$, hence $\varphi(x')=-\infty$. Similarly, if $x\in \partial D^-$, then $I^-(x)\subseteq \varphi^{-1}(+\infty)$.
    
    In local coordinates, it is then easy to check that, whenever $x\in \partial \op{dom}(\varphi)$, there exists an open cone with apex at $x$ contained in $M\backslash \overline{\op{dom}(\varphi)}$. This proves that $\partial \op{dom}(\varphi)$ is countably $n$-rectifiable (similar to \cite{Villani}, Theorem 10.48).
\end{proof}

\begin{figure}[htbp] 
    \centering

\begin{tikzpicture}[scale = 0.8]

    \draw[->] (-5, 0) -- (5, 0); 
    \draw[->] (0, -5) -- (0, 5); 

    \fill (2,2) circle (2pt); 
    \node at (2,2.3) {$y_1$}; 
    \fill (-2, 2) circle (2pt); 
    \node at (-2, 2.3) {$y_0$}; 
    \fill (-1,-1) circle (2pt);
    \node at (-1,-1.3) {$x_0$};
    \fill (-1.3,-0.7) circle (1.5pt);
    \node at (-1.6,-0.9) {$x_k$};

    \draw[black, thick, dashed] (-2,-2) -- (2,2);
    \draw[black, thick,] (-2,2) -- (0,0);
    \draw[black, thick, dashed] (-2,2) -- (2,-2);
    \draw[black, thick] (2,2) -- (0,0);
    \draw[black, thick] (-2,2) -- (-5,-1);
    \draw[black, thick] (2,2) -- (5,-1);

\end{tikzpicture}

\caption{Discontinuity of $c$-convex functions.}
    \label{fig:meine-grafik2}
\end{figure}

\begin{example}\rm\label{remc}
    \begin{enumerate}[(a)]
    \item 
    We provide an example of a $c$-convex function in the two-dimensional Minkowski space which is not continuous on the interior of its domain.
    
    Consider the $c$-convex function
    \[
        \varphi(x):=\max\{-c(x,y_1),a-c(x,y_0)\},\ a\ll 0.
    \]

    Then $\op{int}(\op{dom}(\varphi))=I^-(y_0)\cup I^-(y_1)$ is the region below the graph shwon in Figure \ref{fig:meine-grafik2}. Consider a point $x_0\in \partial J^-(y_1)\cap I^-(y_0)\subseteq \op{int}(\op{dom}(\varphi))$ and a sequence $x_k$ as illustrated in the figure. Since $a\ll 0$, we have
    \[
        \varphi(x_0)=0 \text{ and } \lim_{x_k\to x_0} \varphi(x_k) = a-c(x_0,y_0)\ll 0.
    \]
    This shows that $\varphi$ is not continuous at $x_0$.
    \medskip 
    
    However, $\varphi$ is continuous and even locally semiconvex on $\op{int}(\op{dom}(\varphi))$, except along the two dashed lines (extended to infinity). 
    This visualizes the result of Theorem \ref{thmc}: It is easy to construct and absolutely continuous probability measures $\mu$, supported on a small neighbourhood of $x_0$, such that, if $\pi$ denotes the optimal coupling between $\mu$ and $\nu:=\frac 12 (\delta_{y_0}+\delta_{y_1})$, then $\varphi$ is a $\pi$-solution and even a Kantorovich potential. Theorem \ref{thmc} asserts that, even though $\varphi$ is not locally semiconvex on the entire ball, it is locally semiconvex on an open subset of full $\mu$-measure.
    \item 
    We provide an example of a $c$-convex function $\varphi$ in the two-dimensional Minkowski space and a compact set $K\subseteq \op{int}(\op{dom}(\varphi))$ such that $\partial_c\varphi(K)$ is unbounded, see Figure \ref{fig2}.
    
    Let $(x^0,x^1)$ denote the coordinates in $\R^2$, set
    \[
    \psi:\R^2\to \overbar \R,\ 
    \psi(y):=
    \begin{cases}
    -(y^1)^{-p},& \text{ if } y=\big(\frac 1{y^1},y^1\big),\ y^1>0,
    \\
   -\infty, &\text{ else,}
    \end{cases}
    \]
    and define the $c$-convex function
    \[
    \varphi(x):=\psi^c(x):=\sup_{y\in \R^2} (\psi(y)-c(x,y)).
    \]
    Fix $y_0:=(2,1/2)$. For $x$ in some small neighbourhood $U$ of $0$, we have $x\in I^-(y_0)$, hence $\varphi(x)>-\infty$.
    
    It is easy to check that if $x\in I^-(y)$ for some $y=((y^1)^{-1},y^1)\in \R^2$, then $(y^1)^{-1}>x^0$ and
    \begin{align}
    \psi(y)-c(x,y) 
    \leq -(y^1)^{-p}+d^p((x^0,y^1),y)
    \leq |x^0|^p. \label{eqbr}
    \end{align}
    Hence, $\varphi(x)\in \R$ for $x\in U$, and therefore $0\in \op{int}(\op{dom}(\varphi))$.
    
    Moreover, we have $\psi(y)-c(x,y)=0$ for any $x=(0,x^1)$ and $y=((x^1)^{-1},x^1)$, $x^1>0$. This, together with \eqref{eqbr}, shows that $y\in \partial_c\varphi(x)$. Consequently, $\partial_c\varphi(B_{\ep}(0))$ is unbounded for any $\ep>0$.
    \end{enumerate}
\end{example}

\begin{figure}[htbp]
        \centering 
        \begin{tikzpicture}[scale=1]

            \draw[->] (-3,0) -- (5,0);
            \draw[->] (0,-2) -- (0,5);

            \draw[black, thick] (-4,0) -- (4,0);

            \fill (0.75,1.33) circle (2pt);
            \node at (1.1,1.33) {$y_0$};

            \fill (0.3,0) circle (2pt);
            \node at (0.3,-0.3) {$x$};

            \draw[black, dashed] (0.3,0) -- (0.3,3.33);
            
            \fill (0.3, 3.33) circle (2pt);
            \node at (1.4,3.33) {$y\in \partial_c\varphi(x)$};
            
            \draw (0,0) circle (15pt);
            \node at (-0.5,-0.7) {$U$};

            \foreach \x in {0.2,0.21,...,2}{
                \pgfmathsetmacro\xnext{\x+0.01}
                \pgfmathsetmacro\px{\x}
                \pgfmathsetmacro\py{1/\x}
                \pgfmathsetmacro\npx{\xnext}
                \pgfmathsetmacro\npy{1/\xnext)}
                \draw[black, thick] (\px,\py) -- (\npx,\npy);
            }

        \end{tikzpicture}

    \caption{Unboundedness of the $c$-subdifferential.}

    \label{fig2}
    \end{figure}

\section{Regularity of $c$-convex functions}\label{secc}

Theorem \ref{thmc} is a consequence of Corollary \ref{corasd}, which concerns the regularity of $c$-convex functions.

\begin{theorem}[Structure of the $c$-subdifferential]\label{thmai}
    Let $\varphi$ be a $c$-convex function relatively to $(M,A)$ for some closed set $A\subseteq M$. Let 
    \[
        \Omega := \op{int}(\op{dom}(\varphi))\cap A^c \text{ and } \psi:=\varphi^c.
    \]
    Then for $\op{vol}_g$-a.e.\ $x\in \Omega$ there exists $\delta$ such that
    \begin{align*}
        \sup\{\psi(y)-c(x,y)\mid  y\in A,\ d(x,y)\leq \delta\}< \varphi(x).
    \end{align*}
\end{theorem}
\begin{proof}
    In Subsection \ref{ölkjhgfdsrftz}. 
\end{proof}

\begin{corollary}[Regularity of $c$-convex functions]\label{corasd}
    Consider the hypothesis of the above theorem with $A$ assumed to be causally compact. Define $\Omega_0$ to be the open set of all $x_0\in \Omega$ such that the following holds: There exists $\delta>0$ and a neighbourhood $x_0\in U\subseteq \Omega$ such that, for all $x\in U$,
    \begin{align*}
        \sup\{\psi(y)-c(x,y)\mid y\in A,\ d(x,y)\leq \delta\}<\varphi(x).
    \end{align*}
    Then $\Omega\backslash \Omega_0$ is $\op{vol}_g$-negligible and $\varphi$ is locally semiconvex on $\Omega_0$.
\end{corollary}
\begin{proof}
    Fix $x_0\in \Omega$ for which the result from Theorem \ref{thmai} holds. We claim that
    \[
    \exists \text{ open neighbourhood $U=U(x_0)\subseteq \Omega$ of $x_0$: }\quad U\cap I^+(x_0)\subseteq \Omega_0.
    \]    
    To prove the claim, notice that we find $\delta,\ep>0$ and $y_0\in A\cap I^+(x_0)$ such that
    \begin{align*}
        \sup\{\psi(y)-c(x_0,y)\mid y\in A,\ d(x,y)\leq \delta\}\leq\varphi(x_0)-\ep \leq \psi(y_0)-c(x_0,y_0)-\frac \ep 2.
    \end{align*}
    Now suppose, by contradiction, that there is a sequence $(x_k)\subseteq I^+(x_0)$ with $x_k\to x_0$ and a sequence $(y_k)\subseteq J^+(x_k)\cap A\subseteq J^+(x_0)$ such that $d(x_k,y_k)\leq \frac{1}{k}$ and $\psi(y_k)-c(x_k,y_k)\geq \varphi(x_k)-\frac{1}{k}$.
    By the causal compactness of $A$, the sequence $(y_k)$ is precompact. Then $d(x_0,y_k)\to 0$ by the uniform continuity of $d$ on compact sets.
    Thus, for large $k$, we have $d(x_0,y_k)\leq \delta$ and, hence, 
    \begin{align*}
        \psi(y_k)-c(x_0,y_k)\leq \varphi(x_0)-\ep\leq \psi(y_0)-c(x_0,y_0)-\frac \ep 2.
    \end{align*}
    Thus, for large $k$, it follows from the uniform continuity of $c$ on compact subsets of $J^+$ (note that $y_k\in J^+(x_k)\subseteq J^+(x_0)$ and that $y_0\in J^+(x_k)$ for large $k$) that
    \begin{align*}
        \psi(y_k)-c(x_k,y_k)\leq \psi(y_0)-c(x_k,y_0)-\frac{\ep}{4}\leq \varphi(x_k)-\frac{\ep}{4}.
    \end{align*}
    This is a contradiction, and thus proves the claim.

    We now prove that $\op{vol}_g(\Omega\backslash \Omega_0)=0$. Let $\tilde \Omega\subseteq \Omega$ be the set of all $x_0$ such that the above theorem holds. We define the open set
    \begin{align*}
        \tilde \Omega_{0}:=\bigcup_{x_0\in \tilde \Omega} U(x_0)\cap I^+(x_0)\subseteq \Omega.
    \end{align*}
    Since $\Omega_0\supseteq \tilde \Omega_0$, it suffices to prove that $\op{vol}_g(\Omega\backslash\tilde \Omega_0)=0$. Suppose the contrary. Since $\op{vol}(\Omega\backslash \tilde \Omega)=0$, the $\op{vol}_g$-measurable set $B:=\tilde \Omega\backslash \tilde \Omega_0$ has positive measure, so we can find a Lebesgue point $x_0$ of $B$. But this obviously contradicts
    \[
        U(x_0)\cap I^+(x_0)\subseteq \tilde \Omega_0. 
    \]
    To conclude the proof, we now prove the statement concerning the semiconvexity. Fix $x_0\in \Omega_0$. 
    By definition of $\Omega_0$, there exists an open precompact neighbourhood $U\subseteq \Omega_0$ of $x_0$ and $\delta>0$ such that, for any $x\in U$,
    \begin{align*}
        \varphi(x)=\sup\{\psi(y)-c(x,y)\mid y\in A,\ \ d(x,y)\geq \delta\}.
    \end{align*}
    By the precompactness of $U$ and the causal compactness of $A$, the set $A\cap J^+(\overbar U)$ is compact.
    It is then easy to see that there exists an open neighbourhood $V\subseteq U$ of $x_0$ and a compact set $K\subseteq A$ with $V\times K\subseteq I^+$ such that
    \begin{align*}
        \forall\, x\in V:\ \varphi(x)=\sup\{\psi(y)-c(x,y)\mid y\in K\}.
    \end{align*}
    Since $c$ is well-known to be locally semiconcave on $I^+$ (Theorem \ref{A}), it follows that $\varphi$ is locally semiconvex on $V$ (\cite{Fathi/Figalli}, Corollary A.18). Since $x_0$ was arbitrary, $\varphi$ is locally semiconvex on $\Omega_0$.
\end{proof}

We can now prove Theorem \ref{thmc}.

\begin{proof}[Proof of Theorem \ref{thmc}] 
Let $A:=\supp(\nu)$ and
\[
    \Omega := \op{int}(\op{dom}(\varphi))\cap \supp(\nu)^c.
\]
Since $\mu \ll \op{vol}_g$ and $\supp(\mu)\cap \supp(\nu)=\emptyset$, Lemma \ref{Figalli} implies $\mu(\Omega)=1$. 
Set $\psi:=\varphi^c$ and define $\Omega_0$ as above. Since $\mu\ll \op{vol}_g$, Corollary \ref{corasd} implies $\mu(\Omega_0)=1$ (note that the causal compactness of $\supp(\nu)$ is used here).

It follows at once from the above corollary that (i) $\pi$ is concentrated on the set $(\Omega_0\times M)\cap \partial_c\varphi$, which is a subset of $I^+$, and that (ii) $\varphi$ is locally semiconvex on $\Omega_0$.
\end{proof}

\subsection{Proof of Theorem \ref{thmai}}\label{ölkjhgfdsrftz}

We have to prepare for the proof with a few simple lemmas. The first lemma is taken from \cite{Kell} and will be needed in one step in the proof of Theorem \ref{thmai}.

\begin{lemma} \label{pkspadkpa}
    Let $a,b\in \R$, $\ep>0$ and a Borel measurable set $B\subseteq [a,b]$ be given with ${\cal L}^1(B)\geq \ep (b-a)$. Then for all $k\in \N$ there exists $\{t_i\}_{1\leq i\leq k}\subseteq B$ with $t_1<... <t_k$ and $t_{i+1}-t_i\geq \frac{\ep}{2k}(b-a)$.
\end{lemma}
\begin{proof}
    \cite{Kell}, Lemma 3.3.
\end{proof}

\begin{definition}[Convex sets]\rm \label{convex}
    We call an open set $U\subseteq M$ \emph{convex} if there exists an open set $\Omega\subseteq TM$ such that, for all $x\in U$ the set $\Omega_x:=\{v\in T_xM\mid (x,v)\in \Omega\}$ is star-shaped and $\exp_x:\Omega_x\to U$ is a diffeomorphism.
\end{definition}

\begin{remark}\rm \label{remb} 
    In a convex set, there exists a unique geodesic (up to reparametrization) between any two points that lies entirely within the set.
    It is also well known that every point has arbitrarily small convex neighbourhoods \cite{ONeill}, Chapter 5, proof of Proposition 7.
\end{remark}

\begin{lemma}[An orthonormal frame]\label{frame}
    Let $U$ be a convex set. Then there exist smooth maps $e_i:U\times U\to TU$ ($i=0,...,n)$ with the following properties:
    \begin{enumerate}[(i)]
        \item
        For all $x,y\in U$, the set $\{e_0(x,y),...,e_n(x,y)\}$ is an orthonormal basis of the tangent space $T_yM$ such that $e_0(x,y)$ is timelike.
        \item
        For all $i=0,...,n$ and all $x,y\in U$ the tangent vector $e_i(x,y)$ arises as the parallel transport $V$ along the unique (up to reparametrization) geodesic in $U$ from $x$ to $y$ with $V(0)=e_i(x,x)$.
    \end{enumerate}
\end{lemma}
\begin{proof}
    Choose $x_0\in U$ and an orthonormal basis $e_0,...,e_n$ of $T_{x_0}M$. Now define $e_i(x)$ as the parallel transport of $e_i$ along the (unique) geodesic in $U$ from $x_0$ to $x$. Now define $e_i(x,y)$ as the parallel transport of $e_i(x)$ along the (unique) geodesic in $U$ from $x$ to $y$. For more details, see the appendix in \cite{Metsch1}.
\end{proof}

\begin{corollary}\label{cor}
    Let $U$ and $e_i$ be as above. Let $x,y\in U$ and denote by $\gamma:[0,1]\to U$ the unique geodesic connecting $x$ and $y$.
    Let $\exp_{x}^{-1}(y)=\dot \gamma(0)=:\sum_{i=0}^n \lambda_i e_i(x,x)$. Then it holds
    \begin{align*}
        \exp_y^{-1}(x)=-\dot \gamma(1)=\sum_{i=0}^n -\lambda_i e_i(x,y).
    \end{align*}
\end{corollary}
\begin{proof}
    The first claimed equality follows from the definition of the exponential map. For the second, observe that we can write
    \begin{align*}
        -\dot \gamma(1)=\sum_{i=0}^n \mu_i e_i(x,y) \text{ for } 
        \begin{cases}
            \mu_0=-g_y(-\dot \gamma(1),e_0(x,y)),\ &\text{and }
            \\\\
            \mu_i=g_y(-\dot \gamma(1),e_i(x,y)),\ &i\geq 1.
        \end{cases}
    \end{align*}
    Let $i\geq 1$. By construction of $e_i$, $V(t):=e_i(x,\gamma(t))$ is a parallel vector field along $\gamma$. Hence, since the Levi-Civita connection is compatible with the metric, and since $V$ and $\dot \gamma$ are parallel,
    \begin{align*}
        \frac{d}{dt} g_{\gamma(t)}(-\dot \gamma(t),V(t))
        =
        g_{\gamma(t)}\(-\frac{D\dot \gamma}{dt}(t),V(t)\)+g_{\gamma(t)}\(-\dot \gamma(t),\frac{DV}{dt}(t)\)=0.
    \end{align*}
    This shows that
    \begin{align*}
        \mu_i=g_y(-\dot \gamma(1),e_i(x,y))=g_x(-\dot \gamma(0),e_i(x,x))=-\lambda_i.
    \end{align*}
    Analogously, $\mu_0=-\lambda_0$ and we conclude the proof of the corollary.
\end{proof}

\begin{definition}[Spacelike projection]\rm\label{lkjh}
    Let $U$ and $e_i$ be as above.
    We define the smooth projection
    \begin{align*}
        \pi:TU\to TU,\ \(x,\sum_{i=0}^n \lambda_i e_i(x,x)\)\mapsto \(x,\sum_{i=1}^n \lambda_i e_i(x,x)\).
    \end{align*}
    With some abuse of notation we will also write $\pi$ for the corresponding mapping on the tangent spaces $\pi:T_xM\to T_xM$, $x\in U$.
\end{definition}

The following technical lemma provides some uniform estimates which we will need in the proof of Theorem \ref{thmai}.

\begin{lemma}[Uniform estimates] \label{lemma}
    Let $x_0\in M$.  Then we can find an open neighbourhood $U$ of $x_0$ with the following properties:
    \begin{enumerate}[(i)]
        \item
        $U$ is convex and there exists an open convex set $V$ with $U\Subset V$. 
        \item
        The future pointing causal geodesics which lie in $V$ are the unique (up to reparametrization) maximizing curves connecting their endpoints.
        \item
        Fix some Riemannian metric $\tilde h$ on $TM$, and let $e_i$ be a smooth orthonormal frame as above w.r.t.\ $V$. There exists a constant $C>0$ such that:
        \begin{enumerate}[(1)]
            \item $\op{Lip}^{\tilde h}(\pi_{|T^1U})\leq C$
    
            \item
            $d_h(U,\partial V)> C^{-1}$

            \item 
            For all $x\in U$, $y\in B_{C^{-1}}(x)$ and $j,l=0,...,n$, denoting $\bar e_j:=e_j(x,y)\in T_yM$ and $e_j:=e_j(x,x)\in T_xM$, we have
            \begin{align*}
                |g(\bar e_j,\bar e_l)-
                g(d_{x}\exp_{y}^{-1}(e_j), \bar e_l)|
                \leq
                \frac{1}{2(5n+1)}.
            \end{align*}
        \end{enumerate} 

        \item 
        There exists another constant $\tilde C>0$ such that, for all $x\in U$ and $y\in V$ with $d_h(x,y)=C^{-1}$, we have
        \begin{align*}
            |\exp_x^{-1}(y)|_h\leq \tilde C \text{ and } \sum_{j=0}^n \mu_j^2\geq \tilde C^{-1},
        \end{align*}
        where
        \begin{align*}
            \exp_x^{-1}(y)=\sum_{j=0}^n \mu_j e_j(x,x).
        \end{align*}
        
        \item Let
        \begin{align*}
            &k:\op{dom}(k):=V\times \{(x,v,w)\in T^2V\mid \exp_x(v)\in V\}\to \R,\ 
            \\[10pt]
            &(y,(x,v,w))\mapsto 
            g_y(d_v(\exp_y^{-1}\circ \exp_x)(w),\exp_y^{-1}(\exp_x(v))).
        \end{align*}
        Then there is a small $\ep>0$ such that the set
        \begin{align*}
            \overline B_{C^{-1}}(U)\times \{(x,v,w)\in T^2V\mid \ &|v|_h\leq \ep,\ d_{\tilde h}((x,w),\pi(x',w'))\leq \ep 
            \\
            &\text{ for some } (x',w')\in TU,\ |w'|_h=1\}
        \end{align*}
        is compact and contained in $\op{dom}(k)$. In particular, there exists a modulus of continuity $\omega$ such that
        \begin{align*}
            |k(y,(x,v,w))-k(y,(\tilde x,0,\tilde w))| 
            \leq 
            \omega(|v|_h+d_{\tilde h}((x,w),(\tilde x,\tilde w)))
        \end{align*}
        whenever $(y,(x,v,w))$ and $(y,(\tilde x,0,\tilde w))$ belong to this set.
    \end{enumerate}
\end{lemma}
\begin{proof}
    (i) follows from Remark \ref{remb}, and (ii) follows from Remark \ref{remb}, Theorem \ref{asdfghj} and the strong causality of $M$.  (iii)(1) and (2) follow from $U\Subset V$. For (3), note that the map
    \begin{align*}
        f:V\times V\to \R,\ (x,y)\mapsto g_y(\bar e_j,\bar e_l)-g_y(d_x\exp_y^{-1}(e_j),\bar e_l),
    \end{align*}
    is smooth and $f(x,x)=0$. Thus, the claim follows from the uniform continuity of $f$ on compact subsets. Part (iv) follows immediately from the continuity of the map $V\times V\ni (x,y)\mapsto (x,\exp_x^{-1}(y))$ and from the fact that $B_{C^{-1}}(U)$ is relatively compact in $V$. For the proof of part (v) observe that the compactness (for small $\ep$) follows from the continuity of $\pi$ together with the compactness of the unit tangent bundle over $\overline U$.
\end{proof}

\begin{proof}[Proof of Theorem \ref{thmai}]
Clearly, it suffices to prove the following: If $x_0\in \Omega$, then there exists an open neighbourhood $U\subseteq \Omega$ of $x_0$ such that the statement of the theorem holds for $\op{vol}_g$-a.e.\ $x\in U$. Thus, it suffices to prove the stated property for $\op{vol}_g$-a.e.\ $x\in U$ where $U\Subset \Omega$ is as in the above lemma. For the proof, we fix an orthonormal frame $e:V\times V\to TV$ as in Lemma \ref{frame}, where $V$ is as in the above lemma. Let $\pi:TV\to TV$ be as in Definition \ref{lkjh}.
\medskip 

We argue by contradiction and assume that the set
\begin{align*}
B:=\{x\in U\mid \exists (y_k)_k\subseteq J^+(x)\cap A, d(x,y_k)\to 0, \psi(y_k)-c(x,y_k)\to \varphi(x)\}
\end{align*}
is not a null set. 
\medskip

\noindent \textbf{Step 1: The idea:} To obtain a contradiction, we will use that $\op{vol}_g(B)>0$ implies the existence of a finite sequence $(x_i,y_i)_{1\leq i\leq m}\subseteq U\times A$, $m\in \N$, with $(x_i,y_i),(x_{i+1},y_i)\in J^+$ such that
\begin{align}
\psi(y_i)-\varphi(x_i) \geq -d^p(x_i,y_i)-\frac 1m \label{eqc} \text{ and such that }
\\[5pt]
\varphi(x_1)> 1+\varphi(x_m)+\sum_{i=1}^{m-1} \Big(-d^p(x_{i+1},y_i)+d^p(x_i,y_i)\Big). \label{eqb}
\end{align}
On the other hand, the $c$-convexity of $\varphi$ together with \eqref{eqc} shows that \eqref{eqb} is impossible (\cite{Villani}, idea of proof of Theorem 5.10): Indeed, \eqref{eqc} and the $c$-convexity imply for all $i=1,...,m-1$
\[
    \psi(y_i)-\varphi(x_i) \geq -d^p(x_i,y_i)-\frac 1m \text{ and } \psi(y_i)-\varphi(x_{i+1})\leq c(x_{i+1},y_i).
\]
Note that $c(x_{i+1},y_i)=-d^p(x_{i+1},y_i)$. Thus, substracting the first inequality from the second, and adding over $i$, we obtain 
\[
    \varphi(x_1) \leq 1+\varphi(x_m)+\sum_{i=1}^{m-1} \Big(-d^p(x_{i+1},y_i)+d^p(x_i,y_i)\Big),
\]
a contradiction to \eqref{eqb}.
\medskip

\noindent \textbf{Step 2: Fixing some constants:}
Before we start with the actual proof, let us fix, once and for all, the following constants and functions:
\begin{enumerate}[(I)]
    \item 
    $\tilde h$, $C$, $\tilde C$, $\ep$, $k$, $\omega$ from Lemma \ref{lemma} such that $d_h(U,A)\geq C^{-1}$.
    \item 
    $\delta>0$ such that 
    \[
        (C+2)\, \delta\leq \ep \text{ and } \omega((C+3)\, \delta)\leq \frac{\tilde C^{-2}}{24}.
    \]
    \item $R>0$ such that $|\varphi|\leq R$ on $U$ (Lemma \ref{Figalli}).
\end{enumerate}

\noindent \textbf{Step 3: Construction of the sequence:} If $x\in B$, we can pick a sequence $(y_{x,k})_k$ as specified in the definition of $B$. By definition of $\Omega$, we have $\Omega\cap A=\emptyset$, so $y_{x,k}\neq x$. Pick maximizing geodesics $\gamma_{x,k}:[0,1]\to M$ connecting $x$ to $y_k$, and define 
\[
v_{x,k}:=\frac{\dot \gamma_{x,k}(0)}{|\dot \gamma_{x,k}(0)|_h}.
\]
Recall that $\tilde h$ from Lemma \ref{lemma} is a Riemannian metric on $TM$. Since the unit tangent bundle over $U$, $T^1U\Subset T^1V$, is precompact, we can cover it with finitely many open sets
\[
W_i\subseteq TM,\ i=1,...,N, \text{ with } \op{diam}^{\tilde h}(W_i)\leq \delta.
\]
Then,
\begin{align*}
    B\subseteq \bigcup_{i=1}^N B_i, 
\end{align*}
where
\[
    B_i:=\{x\in B\mid (x,v_{x,k})\in W_i \text{ for infinitely many $k$} \}.
\]
Since $B$ is not a null set, there is some $i=1,....,N$ such that $B_{i}$ is not a null set. Consider the closure $\overline B_{i}$. Since this set is measurable and has positive measure, we can find a Lebesgue point $x_*\in \overline B_i\subseteq \overline U$.

Let us choose a vector $v_*\in T_{x_*}M$ with $|v_*|_h=1$ such that $(x_*,v_*)\in \overline W_{i}$ and define $w_*:=\pi(v_*)$. Since $x_*$ is a Lebesgue point, we 
can apply Fubini's theorem to find a nonzero vector $u_*\in T_{x_*}M$ with $d_{\tilde h}((x_*,u_*),(x_*,w_*))\leq \delta$, along with a constant $\delta_*>0$ such that the following properties hold:
\begin{enumerate}
    \item[(IV)\textsubscript{1}]
    $\delta_* |u_*|_h \leq \delta$,
    
    \item[(IV)\textsubscript{2}]
    The exponential map $\exp_{x_*} : B_{\delta_* |u_*|_h}(0) \to V$ is a diffeomorphism onto its image,
    
    \item[(IV)\textsubscript{3}]
    The set $\{ r \in (0, \delta_*) \mid \exp_{x_*}(r u_*) \in \overline{B_i} \}$ has 1-dimensional Lebesgue measure at least $\delta_*/2$.
\end{enumerate}
Since $p\leq 1$, we can choose $m\in \N$ such that
 \begin{enumerate}
     \item[(V)] $\frac{1}{m^2}\leq \frac{\tilde C^{-1}}{2}$ and $1+2R - \frac{p\, \tilde C^{-2}\delta_*}{96} \left(\frac{1}{m^2}+\frac{\tilde C^{-2}\delta_*}{24m}\right)^\frac{p-2}2<0.$
 \end{enumerate}
 Regarding (IV)\textsubscript{3}, Lemma \ref{pkspadkpa} then yields the existence of
\begin{align}
    0\leq r_1\leq ...\leq r_m< \delta_* \text{ with }
    r_{i+1}-r_i\geq \frac{\delta_*}{4m} \text{ and } \exp_{x_*}(r_iu_*)\in \overline B_i. \label{hjs}
\end{align}
Since the exponential map in (IV)\textsubscript{2} is a diffeomorphism, for each $1\leq i\leq m$, we can find $r_iu_*\approx u_i\in T_{x_*}M$ with
\begin{align}
    u_i\in B_{\delta_*|u_*|_h}(0),\  d_{\tilde h}\((x_*,u_*),\Big(x_*,\frac{u_{i+1}-u_i}{r_{i+1}-r_i}\Big)\)\leq \delta \text{ and } \exp_{x_*}(u_i)\in B_i. \label{eqbh}
\end{align}
 For each $i$, define 
\[
x_i:=\exp_{x_*}(u_i)\in B_i. 
\]
By definition of $B_i$ and $W_i$, there exists $y_i\in J^+(x_i)\cap A$ such that
\begin{align}
     d(x_i,y_i)\leq \frac{1}{m} \text{ and } \psi(y_i)+d^p(x_i,y_i)\geq \varphi(x_i)-\frac{1}{m} \label{jidoadpoa}
\end{align}
and furthermore, there exists a maximizing geodesic $\gamma_i:[0,1]\to M$ connecting $x_i$ to $y_i$ with
\begin{align}
    \(x_i,\frac{\dot \gamma_i(0)}{|\dot \gamma_i(0)|_h}\)\in W_i. \label{ujnuzhbgtfc}
\end{align}

\noindent \textbf{Step 4: Estimating the distances:}
We claim: \emph{$x_{i+1}$ and $y_i$ are always causally related and 
\begin{align*}
    d(x_{i+1},y_i)^2\geq d(x_i,y_i)^2+\frac{\tilde C^{-2}}{6}(r_{i+1}-r_i).
\end{align*}
}

We postpone the proof to the end. 
\\

\noindent \textbf{Step 5: Conclusion:} 
 \eqref{eqc} follows from from \eqref{jidoadpoa}. According to Step 1, it suffices to prove \eqref{eqb}. 
 
 Since $r_{i+1}-r_i\geq \frac{\delta_*}{4m}$ it follows from Step 4 that
\begin{align*}
    d(x_{i+1}, y_i)^2
    \geq
    d(x_i,y_i)^2 + \frac{\tilde C^{-2} \delta_*}{24m}.
\end{align*}
This together with (III) gives
\begin{align*}
    &1+\varphi(x_m)-\varphi(x_1)+\sum_{i=1}^{m-1} \Big(-d^p(x_{i+1},y_i)+d^p(x_i,y_i)\Big)
    \\[10pt]
    \leq
    &\, 1+2R-
    \sum_{i=1}^{m-1} \(\Big(d^2(x_{i},y_i)+\frac{\tilde C^{-2}\delta_*}{24m}\Big)^\frac p2 -d^2(x_i,y_i)^\frac p2\).
\end{align*}
Since
\[
(a+b)^q -a^q \geq  q\, b (a+b)^{q-1} \text{ whenever } 0<q\leq 1 \text{ and } a,b\geq 0,
\]
and since $d(x_i,y_i)\leq \frac 1m$,
 this expression is less or equal to
\begin{align*}
    &1+2R 
     -(m-1) \frac p2\frac{\tilde C^{-2}\delta_*}{24m} \Big(\frac{1}{m^2}+\frac{\tilde C^{-2}\delta_*}{24m}\Big)^\frac{p-2}2 
     \\[10pt]
     &\leq\, 1+2R - \frac{p\, \tilde C^{-2}\delta_*}{96} \left(\frac{1}{m^2}+\frac{\tilde C^{-2}\delta_*}{24m}\right)^\frac{p-2}2<0,
\end{align*}
where in the last step we used (V). This proves \eqref{eqb}, and this is a contradiction according to Step 1.
\bigskip

\noindent \textbf{Proof of Step 4:}
Fix $i=1,...,m-1$. There exists a maximizing geodesic $\gamma_i:[0,1]\to M$ connecting $x_i$ to $y_i$ such that \eqref{ujnuzhbgtfc} holds. Since $y_i\in A$ and $x_i\in U$, (I) implies $d_h(x_i,y_i)\geq C^{-1}$. Let us choose the first $t_i\in [0,1]$ with $d_h(x_i,\gamma_i(t_i))=C^{-1}$. By Lemma \ref{lemma}(iii)(2), we have
\[
\bar y_i:=\gamma_i(t_i)\in V.
\]
Using the reverse triangle inequality and the fact that $\gamma_i$ is maximizing, we obtain
\begin{align*}
d(x_i,y_i)=d(x_i,\bar y_i)+d(\bar y_i,y_i) \text{ and } d(x_{i+1},y_i)\geq d(x_{i+1},\bar y_i)+d(\bar y_i,y_i).
\end{align*}
Thus, if we can show that
\begin{align}
    d(x_{i+1},\bar y_i)^2\geq d(x_i,\bar y_i)^2+\frac{\tilde C^{-2}}{6}(r_{i+1}-r_i) \label{eqrefa}
\end{align}
then we automatically have $d(x_{i+1},\bar y_i)\geq d(x_i,\bar y_i)$ and it also follows
\begin{align*}
    d(x_{i+1},y_i)^2&\geq d(x_{i+1},\bar y_i)^2+2d(x_{i+1},\bar y_i)d(\bar y_i,y_i)+d(\bar y_i,y_i)^2
    \\
    &\geq d(x_i,\bar y_i)^2+\frac{\tilde C^{-2}}{6}(r_{i+1}-r_i).
\end{align*}
In particular, $x_{i+1}$ and $y_i$ are causally related and the claim is proven. Thus, it remains to prove \eqref{eqrefa}.
\\

From Lemma \ref{lemma}(ii) we infer that, if $\exp_{\bar y_i}^{-1}(x_{i+1})\in -\C_{\bar y_i}$, then $x_{i+1}$ and $\bar y_i$ are causally related and 
\[
d(x_{i+1},\bar y_i)=|\exp_{\bar y_i}^{-1}(x_{i+1})|_g.
\]
Therefore, we need to compute $f(x_{i+1},\bar y_i)$ for the function
\begin{align*}
    f:V\times V\to \R,\ f(x,y):=-g(\exp_y^{-1}(x),\exp_y^{-1}(x)).
\end{align*} 
Clearly, $f$ is a smooth map thanks to the convexity of $V$.
Using first order Taylor-expansion of the map $h:=f(\cdot,\bar y_i)\circ \exp_{x_*}$ at the point $u_i\in T_{x_*}M$, we obtain for some $u=(1-t)u_i+tu_{i+1}$ $(t\in (0,1))$\footnote{Note that $((1-s)u_i+su_{i+1})\in B_{\delta_*|u_*|_h}\subseteq \op{dom}(h)\  \forall s\in [0,1]$ thanks to (IV)\textsubscript{2} and \eqref{eqbh}, so that the mean value theorem used from the second to thrid line is applicable.}
\begin{align}
A&:=f(x_{i+1},\bar y_i) -f(x_i,\bar y_i) \nonumber
\\[10pt]
&=h(u_{i+1})-h(u_i) \nonumber
\\[10pt]
&= d_{u}h(u_{i+1}-u_i)\nonumber
\\[10pt]
&= \nonumber
-2\, g\Big(d_{u}\(\exp_{\bar y_i}^{-1}\circ \exp_{x_*}\)(u_{i+1}-u_i),\exp_{\bar y_i}^{-1}(\exp_{x_*}(u))\Big) \nonumber
\\[10pt]
&=-2(r_{i+1}-r_i)\, g\Big(d_{u}\(\exp_{\bar y_i}^{-1}\circ \exp_{x_*}\)\(u_*+\ep_i\),\exp_{\bar y_i}^{-1}(\exp_{x_*}(u))\Big),\nonumber
\\[10pt]
&=-2(r_{i+1}-r_i)\, k(\bar y_i,(x_*,u,u_*+\ep_i)).\label{ikmjuzhgt}
\end{align}
Here, we have set $(r_{i+1}-r_i)(u_*+\ep_i):=(u_{i+1}-u_i)$, and recall that $k$ denotes the map from Lemma \ref{lemma}(v). 

Let $\bar \gamma_i:[0,1]\to V$ be the geodesic reparametrization of the first part of $\gamma_i$ such that $\bar \gamma_i(1)=\bar y_i$. By Lemma \ref{lemma}(ii), $\bar \gamma_i$ is the unique maximizing geodesic (up to reparametrization) connecting $x_i$ to $\bar y_i$. Denote $e_j:=e_j(x_i,x_i)$ so that
\begin{align}
\bar v_i:=\dot {\bar \gamma}_i(0)=\exp_{x_i}^{-1}(\bar y_i)=\sum_{j=0}^n \mu_j e_j \text{ and } \bar w_i:=\pi(\bar v_i)=\sum_{j=1}^n \mu_j e_j \label{eqbk}
\end{align}
for some $\mu_j\in \R$.

\noindent \textbf{Claim:} We have
\[
|k(\bar y_i,(x_*,u,u_*+\ep_i))-k(\bar y_i,(x_i,0,|\bar v_i|_h^{-1}\bar w_i))| \leq \omega((C+3)\delta).
\]

\noindent \textbf{Proof of claim:}
From (IV)\textsubscript{1} and (II),
\begin{align}
|u|_h\leq (1-t)\, |u_i|_h+t\, |u_{i+1}|_h< \delta_*\, |u_*|_h <\delta \leq \ep \label{eqbi}
\end{align}
and by the triangle inequality, the definition of $u_*$ and \eqref{eqbh},
\begin{align*}
&d_{\tilde h}((x_i,|\bar v_i|_h^{-1}\bar w_i)),(x_*,u_*+\ep_i))
\\[8pt]
\leq\ &d_{\tilde h}\(\(x_i,\pi\Big(|\dot {\bar \gamma}_i(0)|_h^{-1}\dot {\bar \gamma}_i(0)\Big)\),(x_*,w_*)\)+d_{\tilde h}((x_*,w_*),(x_*,u_*))+d_{\tilde h}((x_*,u_*),(x_*,u_*+\ep_i)) 
\\[10pt]
=\ &d_{\tilde h}\(\pi\Big(x_i,|\dot \gamma_i(0)|_h^{-1}\dot \gamma_i(0)\Big),\pi(x_*,v_*)\)+2\delta. 
\end{align*}
Since $(x_i,|\dot \gamma_i(0)|_h^{-1}\dot \gamma_i(0)),(x_*,v_*)\in \overline W_{i}\subseteq \overline{T^1U}$ by \eqref{ujnuzhbgtfc}, and $\op{diam}^{\tilde h}(\overline W_{i})\leq  \delta$, we get from Lemma \ref{lemma}(iii)(1) and (II) that
\begin{align}
d_{\tilde h}((x_i,|\bar v_i|_h^{-1}\bar w_i),(x_*,u_*+\ep_i)) \label{eqbj}
\leq 
\op{Lip}(\pi_{|T^1U})\, \delta+2\delta\leq (C+2)\, \delta\leq \ep.
\end{align}
We use Lemma \ref{lemma}(v) together with \eqref{eqbi} and \eqref{eqbj} to conclude the proof of the claim.
\hfill \checkmark
\medskip

We can use the claim, the modulus $\omega$ from Lemma \ref{lemma}(v) and (II) to estimate \eqref{ikmjuzhgt} and we obtain
\begin{align}
    A\geq -2(r_{i+1}-r_i)\, k(\bar y_i,(x_i,0,|\bar v_i|_h^{-1}\bar w_i))-2(r_{i+1}-r_i)\, \frac{\tilde C^{-2}}{24} \label{eqbl}
\end{align}
By definition of $k$, we have
\begin{align}
k(\bar y_i,(x_i,0,|\bar v_i|_h^{-1}\bar w_i))=|\bar v_i|_h^{-1}g(d_{x_i}\exp_{\bar y_i}^{-1}(\bar w_i),\exp_{\bar y_i}^{-1}(x_i)).\label{eqbm}
\end{align}
Writing $\bar e_j:=e_j(x_i,\bar y_i)$, we know from Corollary \ref{cor} and \eqref{eqbk} that
\begin{align}
\exp_{\bar y_i}^{-1}(x_i)=\sum_{j=0}^n -\mu_j \bar e_j.\label{eqbn}
\end{align}
We insert \eqref{eqbm}, \eqref{eqbn} and \eqref{eqbk} into \eqref{eqbl} and obtain
\begin{align*}
    A
    \geq &-2(r_{i+1}-r_i)\, |\bar v_i|_h^{-1}\sum_{\substack{j=1,\\l=0}}^n \mu_j (-\mu_l) g(d_{x_i}\exp_{\bar y_i}^{-1}(e_j),\bar e_l)-(r_{i+1}-r_i)\, \frac{\tilde C^{-2}}{12}.
\end{align*}
Next we use (iii) of Lemma \ref{lemma} to deduce that 
\begin{align*}
    |g(d_{x_i}\exp_{\bar y_i}^{-1}(e_j),\bar e_l)-g(\bar e_j,\bar e_l)|\leq \frac1{2(5n+1)}.
\end{align*} 
Thus,
\begin{align*}
    A\geq 
    &-2(r_{i+1}-r_i)\, |\bar v_i|_h^{-1} \sum_{\substack{j=1,\\l=0}}^n \Big(\mu_j\,  (-\mu_l)\, g(\bar e_j,\bar e_l)\Big)
    \\[10pt]
    &-\frac{(r_{i+1}-r_i)\, |\bar v_i|_h^{-1}}{5n+1}\sum_{\substack{j=1,\\l=0}}^n |\mu_j|\, |\mu_l|-(r_{i+1}-r_i)\, \frac{\tilde C^{-2}}{12}.
\end{align*}
Using that $(\bar e_j)$ is an orthonormal basis in the tangent space $T_{\bar y_i}M$, we obtain
\begin{align*}
    A\geq 
     2(r_{i+1}-r_i)\, |\bar v_i|_h^{-1}\sum_{j=1}^n \mu_j^2 
   &-\frac{(r_{i+1}-r_i)\,|\bar v_i|_h^{-1}}{5n+1} \sum_{\substack{j=1,\\l=0}}^n |\mu_j|\, |\mu_l|
    \\
    &-(r_{i+1}-r_i)\,\frac{\tilde C^{-2}}{12}.
\end{align*}
By \eqref{eqbk}, \eqref{jidoadpoa}, the definition of $m$ and Lemma \ref{lemma}(iv), we have
\begin{align}
    \mu_0^2-\sum_{j=1}^n \mu_j^2=|\exp_{x_i}^{-1}(\bar y_i)|_g^2=d(x_i,\bar y_i)^2\leq \frac{1}{m^2}\leq \frac{\tilde C^{-1}}{2}\leq \frac{1}{2}\sum_{j=0}^n \mu_j^2. \label{eqbo}
\end{align}
 This yields
\begin{align*}
    \sum_{\substack{j=1,\\l=0}}^n |\mu_j|\, |\mu_l|
    \leq
    n\mu_0^2+(2n+1)\sum_{j=1}^n \mu_j^2 
    \leq
    (5n+1)\sum_{j=1}^n \mu_j^2.
\end{align*}
This gives
\begin{align*}
    A
    \geq (r_{i+1}-r_i)\, |\bar v_i|_h^{-1} \sum_{j=1}^n \mu_j^2 -(r_{i+1}-r_i)\, \frac{\tilde C^{-2}}{12}.
\end{align*}
In this inequality we insert both inequalities from Lemma \ref{lemma}(iv); the first one together with $|\bar v_i|_h^{-1}=|\exp_{x_i}^{-1}(\bar y_i)|_h^{-1}$, and the second one together with \eqref{eqbo}. This gives
\begin{align*}
    A
    \geq (r_{i+1}-r_i)\, \frac{\tilde C^{-2}}{4}-(r_{i+1}-r_i)\, \frac{\tilde C^{-2}}{12}
    \geq \frac{\tilde C^{-2}}{6}(r_{i+1}-r_i)> 0.
\end{align*}
Since $f(x_i,\bar y_i)=d(x_i,\bar y_i)^2$, we get
\[
    f(x_{i+1},\bar y_i)\geq d(x_i,\bar y_i)^2+A\geq  d(x_i,\bar y_i)^2 +
    \frac{\tilde C^{-2}}{6}(r_{i+1}-r_i) >0.
\]
Since $V$ is a convex set it follows that either $x_{i+1}< \bar y_i$ or $x_{i+1}> \bar y_i$. This almost proves the claim. However, we still need to argue why $x_{i+1}<\bar y_i$ and not $\bar y_i< x_{i+1}$. 

One can repeat the exact same argument for $x_{i+1}(t):=\exp_{x_*}((1-t)u_i+tu_{i+1})\in V$, $t\in (0,1]$, and prove that
\begin{align*}
-g_{\bar y_i}(\exp_{\bar y_i}^{-1}(x_{i+1}(t)),\exp_{\bar y_i}^{-1}(x_{i+1}(t)))> 0
\end{align*}
for all $t$. By reasons of continuity this shows that either $x_{i+1}(t)<\bar y_i$ for all $t\in (0,1]$ or that $\bar y_i< x_{i+1}(t)$ for all $t\in (0,1]$. Since $x_{i+1}(t)\to x_i< \bar y_i$ as $t\to 0$ we deduce that $x_{i+1}(t)>\bar y_i$ cannot be possible for small $t$. Thus we have proved that $x_{i+1}$ and $\bar y_i$ are causally related and that
\begin{align*}
d(x_{i+1},\bar y_i)^2 = f(x_{i+1},\bar y_i)\geq d(x_i,\bar y_i)^2+\frac{\tilde C^{-2}}{6}(r_{i+1}-r_i).
\end{align*}
This finally proves Step 4.
\hfill \checkmark
\end{proof}

\section{Proof of Corollary \ref{cor2}}\label{secd}

\begin{proof}[Proof of Corollary \ref{cor2}]
    
    Continuity of $d$, lower semicontinuity of $c$, and compactness of $\supp(\mu)$ and $\supp(\nu)$ assures finiteness of $C(\mu,\nu)$ and the existence of an optimal coupling $\pi$. Let $\varphi:M\to \overbar \R$ be any $\pi$-solution (the existence is guaranteed by Theorem 2.8 in \cite{Kell}). The fact that $\pi$ is concentrated on $I^+$ follows from Theorem 2.12 in \cite{Kell} or from Theorem \ref{thmc}. We prove the remaining claims:

    Let $\Omega_0\subseteq M$ be an open set of full $\mu$-measure on which $\varphi$ is locally semiconvex, and let $\Omega_1$ be the subset of all differentiability points of $\varphi$. Since $\mu\ll \op{vol}_g$, we have $\mu(\Omega_1)=1$. Consider the set $\partial_c\varphi\cap (\Omega_1\times M)\cap I^+$, which is of full $\pi$-measure. For any $(x,y)$ in this set, the function
    \[
    x'\mapsto \varphi(x')+c(x',y)
    \]
    attains its minimum at $x'=x$. Since $\varphi$ is differentiable at $x$ and $c(\cdot,y)$ is locally semiconcave in a neighbourhood of $x$ (since $y\in I^+(x)$, see Theorem \ref{A}), it follows as in (10.22) in \cite{Villani} that $c(\cdot,y)$ is differentiable at $x$ and that
    \[
    \nabla \varphi(x) +\nabla_xc(x,y)=0.
    \]
    Hence, $\pi$ satisfies \eqref{eqbp} $\pi$-a.e., and since the cost function satisfies the twist condition on $I^+$ (meaning that $\nabla_x c(x,\cdot)$ is injective on its domain of definition; cf.\ \cite{McCann2}, Lemma 3.1), $\pi$ is induced by the map
    \[
    T(x):=(\nabla_x c(x,\cdot))^{-1}(-\nabla \varphi(x)),
    \]
    cf.\ \cite{Ambrosio/Gigli}, Lemma 1.20.
    To prove that \eqref{eqbp} characterizes $\pi$, choose another coupling $\pi'$ that satisfies the equation $\pi'$-a.e. Then we have, for $\pi'$-a.e.\ $(x,y)$:
    \[
    y = (\nabla_x c(x,\cdot))^{-1}(-\nabla \varphi(x)).
    \]
    Hence, again by \cite{Ambrosio/Gigli}, Lemma 1.20, $\pi'$ is also induced by the map $T$, and it follows that $\pi'=\pi$. Hence, \eqref{eqbp} characterizes the optimal coupling. 

    Uniqueness of $\pi$ follows from the fact that we have started with an arbitrary optimal coupling $\pi$ and we showed that $\pi$ is induced by a map $T$. If $\pi'$ is another optimal coupling, it must be induced by a map $T'$. Now $\frac 12 (\pi+\pi')$ is also optimal, hence it is induced by a map as well. This is possible if and only if $T=T'$ $\mu$-a.e., that is, $\pi=\pi'$.
\end{proof}

\begin{figure}[htbp]
    \centering

    \begin{tikzpicture}

        \draw[->] (-4, 0) -- (4, 0); 
        \draw[->] (0, -4) -- (0, 4); 

        \def\angle{-45}

        \fill[gray!30, rotate=\angle, rounded corners=5pt] (-0.4,0) -- (4,0) -- (4,-1) -- (-0.4,-1) -- cycle;

        \fill[gray!50, rotate=\angle, rounded corners=5pt] (3.5,0) -- (4,0) -- (4,-1) -- (3.5,-1) -- cycle;


        -

        \fill (0, 0) circle (2pt); 
        \node at (0.4, -0.2) {$x_0$}; 
        \fill ({sqrt(8)-0.4},{-sqrt(8)-0.4}) circle (2pt); 
        \node at ({sqrt(8)-0.1},{-sqrt(8)-0.4}) {$x$}; 
        \fill (-1, 1) circle (2pt); 
        \node at (-1, 0.6) {$y_0$}; 
        \fill ({sqrt(8)},{sqrt(8)+1}) circle (2pt); 
        \node at ({sqrt(8)+0.3},{sqrt(8)+1}) {$y_1$}; 
        \fill ({sqrt(8)},{-sqrt(8)}) circle (2pt); 
        \node at ({sqrt(8)+0.3},{-sqrt(8)}) {$x_1$}; 

        \node at (0.2,-2) {$\mu$};
    
        \draw[black, thick, -{Latex[length=2mm, width=2mm]}] (0,0) -- (-1,1);
        \draw[black, thick, -{Latex[length=2mm, width=2mm]}] ({sqrt(8)-0.4},{-sqrt(8)-0.4}) -- ({sqrt(8)},{sqrt(8)+1});
    
    \end{tikzpicture}
    \caption{Optimal couplings are not supported on $I^+$.}
    \label{fig:meine-grafik}
\end{figure}

\begin{example}\rm\label{exa}
    We provide an example of two measure $\mu$, $\nu$ satisfying all the assumptions from Corollary \ref{cor2} but the optimal coupling is not supported on $I^+$. 
    
    Consider the two-dimensional Minkowski space and let $x_0:=(0,0)$, $y_0:=(1,-1)$. Consider a rotated rectangle $R$ (as in Figure \ref{fig:meine-grafik}) with upper right corner at $x_1:=(-3,3)$, upper left corner at $(-\ep,\ep)$ and corners rounded off smoothly. Let $y_1:=(4,3)$. Then we have $\supp(\mu)\subseteq I^-(y_1)$ and every point $x\in \supp(\mu)$ lies in $I^-(y_0)$ unless $x$ belongs to the upper edge of $R$. 

    Now let $g:\R^2\to \R$ be a smooth function such that $\{g>0\}=\op{int}(R)$ and $\int g(x)\, dx=1$. Moreover, suppose that 
    \[
        \int_{R'} g(x)\, dx>\frac 12,
    \]
    where $R'$ denotes the dark grey rectangle on the right-hand side of $R$. Let us define the probability measures 
    \[
        \mu:= g\, dx \text{ and } \nu:=\frac 12 (\delta_{y_0}+\delta_{y_1}).
    \]

    It is easy to verify that $(\mu,\nu)$ satisfies all the assumption of Corollary \ref{cor2} (a strictly timelike coupling is obtained by transporting an upper strip of $R$ to $y_1$, and the remaining part of $R$ to $y_0$). However, if $\pi$ denotes the (unique) optimal coupling, let us show that $(x_0,y_0)\in \supp(\pi)$ – in particular, $\supp(\pi)\not \subseteq I^+$.

    To see this, note first that $\supp(\pi)$ must be $c$-cyclically monotone, since $c$ is continuous on $J^+$ and $\supp(\mu)\times \supp(\nu)\subseteq J^+$ (cf.\ \cite{Villani}, Theorem 5.10(ii)).

    Suppose, for contradiction, that $(x_0,y_0)\notin \supp(\pi)$. Then necessarily $(x_0,y_1)\in \supp(\pi)$. By the assumptions on $g$, there must exist some $x\in R'$ which is transported to $y_0$ (that is, $(x,y_0)\in \supp(\pi)$). From the $c$-cyclical monotonicity, we obtain
    \begin{align}
        c(x_0,y_1)+c(x,y_0) \leq c(x_0,y_0)+c(x,y_1)=c(x,y_1). \label{eqd}
    \end{align}
    However, $c$ is continuous on $J^+$, and
    \[
    c(x_0,y_1)+c(x_1,y_0)= -7^\frac p2-0,
    \]
    whereas
    \[
        c(x_1,y_1)=-4^p.
    \]
    Thus, since $x$ is very close to $x_1$, \eqref{eqd} cannot hold. This is a contradiction. Hence, $(x_0,y_0)\in \supp(\pi)$, and therefore $\supp(\pi)$ is not supported on $I^+$.  
\end{example}

\section*{Acknowledgements}

I would like to thank both of my advisors, Markus Kunze and Stefan Suhr, for their helpful guidance during my PhD studies so far, and especially Stefan Suhr for the fruitful discussions about the results of this paper.

\section*{Statements and Declarations}

 \noindent \textbf{Conflict of interest:} None

\noindent \textbf{Data availability:} Apart from the references, this article does not use any external data.

\bibliography{OTaR_090625}
\end{document}